\documentclass[a4paper,10pt]{article}
\pdfoutput=1
\usepackage[utf8]{inputenc}
\usepackage{amsmath,amsthm, amssymb, latexsym}
\usepackage{graphicx}
\usepackage{subfigure}
\usepackage{multirow,booktabs,rotating,url}  
\usepackage{textcomp} 
\usepackage{color}
\usepackage{float}
\usepackage{soul}
\newtheorem{lemma}{Lemma}[section]

\title{Uniform Penalty inversion of two-dimensional NMR Relaxation data}
\author{V. Bortolotti \thanks{ Department of Civil, Environmentals and Materials Engineering (DICAM), University of Bologna}
 \and  R. J. S. Brown\thanks{900  E  Harrison Ave, Apt B-9 Pomona  CA  91767-2024, USA}
 \and  P. Fantazzini\thanks{Department of Physics and Astronomy, University of Bologna; Museo Storico della Fisica e Centro di Studi e Ricerche Enrico Fermi, Piazza del Viminale 1, 00184 Roma – Italy}
 \and G. Landi\thanks{Department of Mathematics, University of Bologna}
 \and F. Zama\footnotemark[4]}%
\date{}
\theoremstyle{definition}
  \newtheorem{definition}{Definition}[section]
\begin{document}
\maketitle
\def\be{\mathbf{e}}
\def\bf{\mathbf{f}}
\def\bs{\mathbf{s}}
\def\bF{\mathbf{F}}
\def\bS{\mathbf{S}}
\def\bC{\mathbf{C}}
\def\bP{\mathbf{P}}
\def\bL2{\mathbf{L}}
\def\bK{\mathbf{K}}
\def\bLambda{\boldsymbol{\Lambda}}
\begin{abstract}
The inversion of two-dimensional NMR  data is an ill-posed problem related to the numerical computation of the inverse Laplace transform.
In this paper we present the 2DUPEN algorithm that extends the Uniform Penalty (UPEN) algorithm
[Borgia, Brown, Fantazzini, {\em Journal of Magnetic Resonance}, 1998] to two-dimensional data.
The  UPEN algorithm, defined for the inversion of one-dimensional NMR relaxation data,  uses Tikhonov-like regularization and
optionally non-negativity constraints in order to implement locally adapted regularization.
In this paper, we analyze the regularization properties of this approach.
Moreover, we extend the one-dimensional UPEN algorithm to the two-dimensional case and present an
efficient implementation based on the Newton Projection method.
Without any a-priori information on the noise norm, 2DUPEN
automatically computes the locally adapted regularization parameters and the distribution of the
unknown NMR parameters by using variable smoothing. Results of numerical experiments on simulated and real
 data are presented in order to illustrate the potential of the proposed method in
reconstructing peaks and flat regions with the same accuracy.
\end{abstract}
%
%
%
%
\section{Introduction}
%
%

Nuclear Magnetic Resonance (NMR) relaxation of $^1$H nuclei measurements is an important tool to analyze the structure of porous media,
ranging from biological systems to hydrocarbon bearing sedimentary rocks \cite{DTH89,HCHB10,Galvosas2007497}.
NMR relaxometry, but also Magnetic Resonance Imaging, is characterized by two relaxation parameters, the longitudinal relaxation
time ($T_1$) and the
transverse relaxation time ($T_2$). When porous media saturated with water or other $^1$H containing fluid are analyzed, $T_1$ and $T_2$ show distributions of relaxation times.

The inversion of two-dimensional NMR relaxation data requires the solution of a first-kind Fredholm integral equation with
separable exponential kernel, occurring in two dimensional inverse Laplace transforms (ILT).
The properties and drawbacks of ILT are analyzed in \cite{IsVy99} and some progress on numeric ILT can be found in \cite{Brancik07, Kuhlman13},
studied only on test problems and not applied to real data.
The most common methods to deal with the  ill-posedness of this problem are the $L_2$ norm regularization methods.
In particular the Tikhonov regularization method has proven to be suitable to the problem of the inversion of NMR data whose solution,
representing a distribution of relaxation times, is usually positive and presents peaks of different heights over flat areas.
Therefore, starting around the year 2000,
the Tikhonov regularization method is applied to reconstruct 2D maps from NMR data in several  papers \cite{VYH02, SVHFFS2002, Moody2004}.

It is well known that the main difficulty of the Tikhonov method is the estimation of the value of the regularization parameter.
Furthermore it is observed that a single regularization parameter does not allow one to reconstruct  peaks and flat regions with the same accuracy.
In this context, the papers \cite{Borgia1998, Borgia2000} propose an algorithm for 1D NMR inversion, based on Tikhonov
regularization with locally adapted regularization parameters and
optional nonnegative constraint.
Such papers state, for the first time, the {\em Uniform PENalty (UPEN) principle}
according to which the product of the  regularization parameter and the curvature value in
each point of the relaxation times distribution should be constant.  Following
the procedure proposed in  \cite{Borgia1998, Borgia2000} the local value of the regularization parameter is computed
combining an estimate  of the residual norm and the local curvature value.
This procedure  has been successfully used to invert 1D NMR data, and a commercial software ({\tt UpenWin} available at  \cite{UPwin}) is
currently used in 1D NMR inversion.

Locally adapted regularization has been proposed in the literature
for image denoising and deblurring by using  regularization terms related to the Total Variation function (see \cite{Gilboa2006,MGrasmair09}
and references therein). Moreover in \cite{HintIP2010,DongHint2011} the local regularization parameters are updated using
the  local filtered residual as estimator of the noise variance.
Recently, $L_1$ sparsity preserving regularization is applied to NMR data in \cite{CMRA2013}.
In that paper a scalar regularization parameter is used to solve
1D problems with low-resolution NMR data with positivity constraint. In this case sophisticated optimization tools are needed, such as interior-point methods.

In this paper, we show that the UPEN principle may be a suitable criterion for choosing the regularization parameters in multiple-parameter Tikhonov regularization. 
Motivated by the good results obtained  in 1D NMR reconstructions, this work extends the UPEN principle  to  two-dimensional
NMR data.
This allows us to compute local values of the regularization parameters related to the curvature in each point of the distribution.
Therefore we obtain an iterative algorithm where, at each iteration, the locally adapted regularization parameters are
automatically updated and an approximate solution is obtained by solving a $L_2$ regularized least squares
problem  subject to lower-bound constraints.
The Newton Projection method \cite{ber82Siam,ber99} is used for the solution of the constrained subproblem.
Such a method proves to be extremely efficient to compute, at each iteration, very accurate solutions.

To the best of our knowledge there are no other papers addressing the extension of the UPEN principle to 2D data, except \cite{GR2012},
where the combination of Tikhonov and UPEN regularization are applied to NMR multidimensional data, without any constraint.
However the application of such a method to real data is quite complicated, due to  the large number of parameters to be set.

The main contribution of this work is twofold. Firstly, it analyzes the regularization properties of the solutions to Tikhonov problem with multiple parameters satisfying the UPEN principle. Secondly, this work extends the UPEN method to multidimensional data and defines an iterative procedure for the automatic computation of the regularized solution and the local regularization parameters.

%

The rest of the paper is organized as follows: in section \ref{PD} we describe the problem.
Section \ref{UP} presents the UPEN principle for the update of the
locally adapted regularization parameters.
The details about the algorithm are discussed in section \ref{Method}.
Numerical results obtained with  synthetic and real data are reported in section \ref{Num},
conclusions are given in section \ref{Concl}.

%
%
%
\section{Problem Description \label{PD}}
NMR data are commonly represented by a signal
measured at different sampling points  which are often evolution times,
but they can be any variable parameter in an experiment, such as excitation frequency,
magnetic field, or field gradient strength.
We consider here 2D NMR relaxation data acquired using a conventional Inversion-Recovery (IR) experiment detected by a
Carr-Purcell-Meiboom-Gill (IR-CPMG) pulse train \cite{Blumich2005}.
The evolution time $t_1$ in IR  and the evolution time $t_2$ in CPMG are two independent variables
and the NMR relaxation data can be written as a two-dimensional array:
\begin{equation}
  S(t_1,t_2)=\iint_0^\infty k_1(t_1,T_1)k_2(t_2,T_2)F(T_1,T_2) \ dT_1 \ dT_2 + e(t_1,t_2) .
\label{model}
\end{equation}
The model equation \eqref{model} is a first-kind Fredholm integral equation whose
kernel is represented by the product of the functions $k_1(t_1,T_1) = 1-2\exp(-t_1/T_1)$,
and $k_2(t_2,T_2)=\exp(-t_2/T_2)$.
The function $e(t_1,t_2)$ represents additive noise, commonly modeled by a Gaussian distribution.
Finally the unknown $F(T_1,T_2)$ is the distribution of  longitudinal and transverse relaxation times.
For all $T_1, T_2$, such distribution is known to be $F(T_1, T_2) \geq \rho$ where $\rho\in\mathbb{R}$.
In this work, we consider $\rho=0$ but we stress that, for some kind of sample, it could be $\rho\neq 0$ and
our algorithm can be easily extended to handle this case.
%

Problem \eqref{model} is discretized by considering $M_1 \times M_2$ samples of the times $t_1$, $t_2$ and by organizing the discrete
observations $\bS \in \mathbb{R}^{M_1 \times M_2}$ in a vector $\mathbf{s}\in \mathbb{R}^{M}$, $M=M_1 \times M_2$.
The unknown discrete distribution $\mathbf{F} \in \mathbb{R}^{N_x \times N_y}$ is obtained by sampling $F$
at $N_x \times N_y$ relaxation times $T_1$ and $T_2$    and it is organized in a vector $\bf \in \mathbb{R}^{N}$, $N=N_x \times N_y$.
Problem \eqref{model} is discretized as:
\begin{equation}
 \mathbf{K} \mathbf{f} + \be = \mathbf{s}
\label{eq:ls}
\end{equation}
where the matrix $\mathbf{K}$ is the Kronecker product
\begin{equation}\label{eq:kronecker}
    \mathbf{K}=\bK_2\otimes \bK_1
\end{equation}
of the matrices $\bK_1 \in \mathbb{R}^{M_1 \times N_x}$
and $\bK_2 \in \mathbb{R}^{M_2 \times N_y}$  obtained
by discretization of the functions $k_1$ and $k_2$ in $M_1 \times N_x$ and $M_2 \times N_y$ points respectively.
The vector $\be\in \mathbb{R}^{M}$ is the discretization of the noise function $e(t_1,t_2)$.

\section{The Uniform Penalty Principle \label{UP}}
The linear system \eqref{eq:ls} is a well-known ill-conditioned inverse problem whose solution is extremely sensitive to the noise.
In order to recover meaningful approximations to the discrete distribution $\mathbf{f}$, some form of regularization is necessary. A commonly implemented regularization strategy is the Tikhonov method that replaces \eqref{eq:ls} by the minimization problem
\begin{equation}\label{eq:tikh0}
 \min_{\mathbf{f}} \left\lbrace \| \mathbf{K} \mathbf{f} - \mathbf{s} \|^2 +  \alpha \|\mathbf{L} \mathbf{f} \|^2 \right\rbrace
\end{equation}
where $\| \cdot \|$ is the $L_2$ norm,  $\mathbf{L}\in\mathbb{R}^{N\times N}$ is the discrete Laplacian operator and $\alpha>0$ is the regularization parameter balancing data fidelity and solution smoothness.

The Tikhonov regularization \eqref{eq:tikh0} requires one to choose a suitable value of the regularization parameter $\alpha$. This is a crucial and difficult task since an universal method does not exist that gives the best value of the regularization parameter for any application \cite{bertero1998}. Assuming one has suitable bounds on the fidelity and regularization terms of the exact solution $\mathbf{f}^*$, i.e:
\begin{equation}\label{miller1}
    \|\mathbf{K} \mathbf{f}^* - \mathbf{s} \|^2 = \varepsilon^2, \qquad \|\mathbf{L} \mathbf{f}^* \|^2 = E^2 ,
\end{equation}
Miller  \cite{miller1970} proposes to set
\begin{equation}\label{miller2}
\alpha = \frac{\varepsilon^2}{E^2}
\end{equation}
and shows that the solution $\mathbf{f}_\alpha$ of \eqref{eq:tikh0}, obtained with the value \eqref{miller2}, satisfies the following conditions
\begin{equation}\label{}
    \|\mathbf{K} \mathbf{f}_\alpha - \mathbf{s} \|^2 \leq \varepsilon^2, \qquad \|\mathbf{L} \mathbf{f}_\alpha \|^2 \leq E^2 .
\end{equation}
As a consequence, at the regularized solution $\mathbf{f}_\alpha$, we have
\begin{equation}\label{}
    \|\mathbf{K} \mathbf{f}_\alpha - \mathbf{s} \|^2 + \alpha\|\mathbf{L} \mathbf{f}_\alpha \|^2 \leq 2\varepsilon^2.
\end{equation}
%
When $\alpha$ is selected such that the fidelity and regularization terms are comparable, the bias is minimized and the result is stable in the presence of noise. However, Tikhonov regularization usually gives distorted solutions with undesired peaks even when $\alpha$ is optimally chosen. 

In order to avoid unwanted peaks and, at the same time, recover the desired ones, multiple-parameter Tikhonov regularization can be used which replaces \eqref{eq:ls} by the minimization problem
\begin{equation}\label{eq:tikh}
     \min_{\mathbf{f}}  \left\lbrace  \| \mathbf{K} \mathbf{f} - \mathbf{s} \|^2 +  \sum_{i=1}^N \lambda_i(\mathbf{L}\mathbf{f})_i^2  \right\rbrace\\
\end{equation}
where $(\mathbf{L}\mathbf{f})_i$ is the $i$-th element of the vector $\mathbf{L}\mathbf{f}$. Now instead of a single regularization parameter $\alpha$, we have $N$ regularization parameters $\lambda_i$, one for each point of the distribution $\mathbf{f}$.
The UPEN method uses the following Uniform Penalty Principle to define the value of each regularization parameter $\lambda_i$.
%
%
\begin{definition}[Uniform Penalty Principle] \label{eq:UPP2} Choose the regularization parameters $\lambda_i$
of multiple-parameter Tikhonov regularization \eqref{eq:tikh} such that, at a solution $\mathbf{f}$,
the terms $\lambda_i(\mathbf{L}\mathbf{f})_i^2$ are constant for all $ i=1,\ldots,N$ such that
    $(\mathbf{L}\mathbf{f})_i^2\neq0$, i.e:
\begin{equation}\label{eq:UPP2}
    \lambda_i(\mathbf{L}\mathbf{f})_i^2 = c, \quad \forall \; i=1,\ldots,N \quad \text{s.t.} \quad (\mathbf{L}\mathbf{f})_i^2\neq0
\end{equation}
where $c$ is a positive constant.
\end{definition}
Observe that, if the non-null terms $\lambda_i(\mathbf{L}\mathbf{f})_i^2$ have all the same constant value,
the regularization parameter $\lambda_i$ is inversely proportional to $(\mathbf{L}\mathbf{f})_i^2$,
so that the value $\lambda_i$ is smaller when $\mathbf{f}$ has fast changes and oscillations and $\lambda_i$ is larger in smooth and flat regions of $\mathbf{f}$. Hence, regularization is enforced in points where the distribution $\mathbf{f}$ is smooth.
%
The following lemmas prove the basic properties of the UPEN principle as a parameter selection rule.
\begin{lemma}
If $\mathbf{f}$ satisfies $\|\mathbf{K} \mathbf{f} - \mathbf{s} \|^2\leq\varepsilon^2$ and the UPEN principle holds with
\begin{equation}\label{eq:C}
    c=\frac{\varepsilon^2}{N_0}
\end{equation}
where $N_0$ is the number of non null terms $(\mathbf{L}\mathbf{f})_i^2$, then
\begin{equation}\label{eq:2epsilon}
    \|\mathbf{K} \mathbf{f} - \mathbf{s} \|^2  + \sum_{i=1}^N \lambda_i(\mathbf{L}\mathbf{f})_i^2 \leq 2\varepsilon^2.
\end{equation}
Conversely, any $\mathbf{f}$ satisfying \eqref{eq:2epsilon} and the UPEN principle with \eqref{eq:C}, also satisfies $\|\mathbf{K} \mathbf{f} - \mathbf{s} \|^2\leq\varepsilon^2$.
\end{lemma}
\begin{proof}
Let $\mathbf{f}$ be such that $\|\mathbf{K} \mathbf{f} - \mathbf{s} \|^2\leq\varepsilon^2$, then, if the UPEN principle is satisfied with \eqref{eq:C},  we have
\begin{equation}
    \|\mathbf{K} \mathbf{f} - \mathbf{s} \|^2  + \sum_{i=1}^N \lambda_i(\mathbf{L}\mathbf{f})_i^2 \leq
    \varepsilon^2  + \sum_{i=1}^{N_0} \frac{\varepsilon^2}{N_0} = 2\varepsilon^2.
\end{equation}
Conversely, if \eqref{eq:2epsilon} and the UPEN principle with \eqref{eq:C} hold, then
\begin{equation}
    2\varepsilon^2 \geq \|\mathbf{K} \mathbf{f} - \mathbf{s} \|^2  + \sum_{i=1}^N \lambda_i(\mathbf{L}\mathbf{f})_i^2 =
    \|\mathbf{K} \mathbf{f} - \mathbf{s} \|^2   + \sum_{i=1}^{N_0} \frac{\varepsilon^2}{N_0} = \|\mathbf{K} \mathbf{f} - \mathbf{s} \|^2+\varepsilon^2.
\end{equation}
\end{proof}
This result shows that the solution $f_{\boldsymbol\lambda}$ of problem \eqref{eq:tikh}, where each component $\lambda_i$ of $\boldsymbol\lambda$ is chosen according to the UPEN principle, is feasible with respect to the data-fidelity constraint $\|\mathbf{K} \mathbf{f} - \mathbf{s} \|^2\leq \varepsilon^2$.

The following lemma shows that $f_{\boldsymbol\lambda}$ is a regularized solution of \eqref{eq:ls}.
\begin{lemma}
Let us define the operator $R_{\boldsymbol\lambda}$ as
\begin{equation}\label{}
    R_{\boldsymbol\lambda} = (\mathbf{K}^T\mathbf{K} +\mathbf{L}^T\mathbf{D}\mathbf{L})^{-1}\mathbf{K}^T
\end{equation}
where $\mathbf{D}$ is the diagonal matrix with diagonal elements
\begin{equation}\label{}
    D_{i,i}=\left\{
             \begin{array}{ll}
               \lambda_i^2, & \hbox{if $(\mathbf{L}\mathbf{f})_i\neq0$;} \\
               \gamma\varepsilon^2, & \hbox{otherwise;}
             \end{array}
           \right.
\end{equation}
where $\gamma$ is a positive constant and the $\lambda_i$ are chosen according to the UPEN principle \eqref{eq:C}, then
\begin{equation}\label{}
    \lim_{\varepsilon\rightarrow 0} R_{\boldsymbol\lambda}\mathbf{K}\mathbf{f}=\mathbf{f}.
\end{equation}
\end{lemma}
\begin{proof}
We observe that, from \eqref{eq:UPP2} and \eqref{eq:C}, we obtain the following expression for the $\lambda_i$:
\begin{equation}\label{eq:lambda}
    \lambda_i = \frac{\varepsilon^2}{N_0(\mathbf{L}\mathbf{f})_i^2} \quad \text{for all} \quad i=1,\ldots,N  \quad \text{such that}
    \quad (\mathbf{L}\mathbf{f})_i^2\neq0 .
\end{equation}
Hence, the proof immediately follows since $\lim_{\varepsilon\rightarrow 0} D_{i,i}=0$ for all $i$.
\end{proof}
Since the regularization parameters $\lambda_i$ defined in \eqref{eq:lambda} 
depend on $f_{\boldsymbol\lambda}$ and $\varepsilon$, which are unknown, 
we propose the following iterative scheme that, given an initial guess $\mathbf{f}^{(0)}$, computes both a solution to \eqref{eq:tikh} and suitable values of the regularization parameters $\lambda_i$, approximately satisfying the UPEN principle.

\smallskip
\paragraph{Iterative scheme.}\hfill\\
\vspace*{-.5cm}
  \begin{description}
     \item[Step 1] Compute $\lambda_i^{(k)} = \displaystyle{\frac{\|\mathbf{K} \mathbf{f}^{(k)} - \mathbf{s} \|^2}{N_0^{(k)}(\mathbf{L}\mathbf{f}^{(k)})_i^2}}$ where $N_0^{(k)}$ is the number of non null terms $(\mathbf{L}\mathbf{f}^{(k)})_i^2$ ;
     \item[Step 2] Compute $\mathbf{f}^{(k+1)}$ by solving \eqref{eq:tikh} with $\lambda_i=\lambda_i^{(k)}$;
     \item[Step 3] Set $k=k+1$.
  \end{description}
\smallskip
In this scheme, the $k$-th residual norm $\|\mathbf{K} \mathbf{f}^{(k)} - \mathbf{s} \|$ is used as an approximation of $\varepsilon$ that, in case of noisy data, is the noise norm $\|\mathbf{e}\|$.

%

Observe that, when one of the terms $(\mathbf{L}\mathbf{f}^{(k)})_i$ in Step 1 is negligible,
it is not possible or not meaningful to make $\lambda_i$
large enough to maintain a truly uniform penalty at such points.
Moreover, a term $(\mathbf{L}\mathbf{f}^{(k)})_i$ could be equal to zero in non flat regions due to noise and
approximation errors generated throughout the iterations. Therefore, in order to have more trustworthy information about the shape of the unknown distribution, it may be advisable to relax the strict uniform-penalty requirement by considering, in the selection rule for the parameters, both second order and first-order
derivative information in a neighborhood of the $i$-th point.
Let us define  the  ${N_x \times N_y}$ matrix $\mathbf{C}$ 
so that lexicographically reordering its elements gives the vector $\mathbf{L}\mathbf{f}$. Moreover, denoted by $\mathbf{P}$ the  ${N_x \times N_y}$ matrix with elements $P_{\ell,\mu}= \|  \nabla \mathbf{F}_{\ell,\mu} \|$ and by $\mathbf{c}$ and $\mathbf{p}$ the $N$ vectors obtained by reordering the elements of $\mathbf{C}$ and $\mathbf{P}$, 
we propose to choose the regularization parameters $\lambda_i^{(k)}$ according to the following relaxed UPEN principle:
\begin{equation}\label{eq:lambda2}
    \lambda_i^{(k)} = \frac{\|\mathbf{K} \mathbf{f}^{(k)} - \mathbf{s} \|^2}{N\left ( \beta_0 +\beta_p
     \underset{\substack{\mu \in I_i}}\max \, (\mathbf{p}^{(k)}_{\mu})^2
 + \beta_c \underset{\substack{\mu \in I_i }} \max \, (\mathbf{c}^{(k)}_{\mu})^2\right )},  \quad i=1,\ldots,N
\end{equation}
%
where the $I_i$ are the indices subsets related to the neighborhood of the pixel $i$ and the $\beta$'s are positive parameters.
The parameter $\beta_0$ prevents division by zero and is a compliance floor,
which should be small enough to prevent undersmoothing, and large enough to avoid oversmoothing.
The optimum value of $\beta_0$, $\beta_c$ and $\beta_p$ can substantially change with the nature of the measured sample.
Therefore, their general optimum value can be only evaluated on the basis of statistical evaluation,
that will be the subject of future research.
The regularization parameters obtained by \eqref{eq:lambda2} are locally adapted: the selection of the values $\lambda_i$ is based on local information about the shape of the desired solution.

%
%
%
\section{The Uniform Penalty Method \label{Method}}
%
%
%
%
In this section, we present an iterative procedure that, in absence of prior information about either the noise norm or the solution smoothness, automatically computes both
the locally adapted regularization parameters $\lambda_i$ and an approximation to
the unknown distribution of relaxation times $\bf$.
We will refer to the proposed method as 2DUPEN since it uses the relaxed UPEN principle to determine the values of regularization parameters. As observed in Section \ref{PD}, the distribution $\bf$ usually satisfies the physical bound $\bf\geq \rho$, $\rho\in\mathbb{R}$; in particular, we consider the usual case $\rho=0$. Hence, the modified Tikhonov problem  is
\begin{equation}\label{eq:tikh2}
    \begin{split}
     & \min_{\mathbf{f}}  \left\lbrace  \| \mathbf{K} \mathbf{f} - \mathbf{s} \|^2 +  \sum_{i=1}^N \lambda_i(\mathbf{L}\mathbf{f})_i^2  \right\rbrace\\
    & \text{ s.t.}  \; \mathbf{f} \geq 0
    \end{split}
\end{equation}

%

The iterative scheme of Section \ref{UP} needs a suitable initial guess $\bf^{(0)}$ which should have a residual norm $\|\mathbf{K} \mathbf{f}^{(0)} - \mathbf{s} \|$ close to the noise norm $\|\mathbf{e}\|$. In \cite{Borgia1998, Borgia2000} this is obtained by means of statistical noise estimation procedures. Here we choose to exploit the regularization properties of the Gradient Projection (GP) method \cite{ber99,Cornelio2013} and define $\bf^{(0)}$ as an over-smoothed approximate solution of the nonnegatively constrained least squares problem
\begin{equation}\label{nnls}
\begin{split}
   &\min\limits_{\mathbf{f}} \left\lbrace  \| \mathbf{K} \mathbf{f} - \mathbf{s} \|^2 \right\rbrace \\
   & \text{ s.t.}  \; \mathbf{f} \geq 0
   \end{split}
\end{equation}
obtained by applying a few iterations of GP.

In order to turn the iterative scheme of Section  \ref{UP}  into a practical algorithm, we need to define a numerical strategy for the solution of \eqref{eq:tikh2}.
A wide numerical experimentation shows that, in order to preserve the relaxed UPEN principle,
high precision is needed in the numerical solution to \eqref{eq:tikh2}. With this aim, second-order methods are preferable to first-order ones,
due to their better convergence characteristics.
Therefore, we consider the Newton Projection (NP) method \cite{ber82Siam,ber99} to solve the constrained minimization problem \eqref{eq:tikh2}. NP is a scaled gradient projection-like method where only the variables in the working set are scaled by the inverse of the corresponding submatrix of the Hessian.  Therefore, the computation of the search direction requires the solution of a linear system at each iteration. 
Local superlinear convergence of the NP method can be proved \cite{ber82Siam}.
Exploiting the structure of the matrix $\mathbf{K}$, we solve the linear system of NP using the Conjugate Gradient (CG) method because the matrix-vector products can be performed efficiently.
In fact, the matrix $\mathbf{K}$ can be represented as a Kronecker product \eqref{eq:kronecker} and matrix-vector products can be performed without ever constructing $\mathbf{K}$
 by using the relation
\begin{equation}\label{Kmat}
    \mathbf{K}\mathbf{x} = \texttt{vec}\big(\mathbf{K}_1\mathbf{X}\mathbf{K}_2^T \big),  \ \ \mathbf{x}=\texttt{vec}(\mathbf{X})
\end{equation}
where $\mathbf{X}\in\mathbb{R}^{N_x \times N_y}$ and, in general, $\texttt{vec}(\mathbf{V})$ is the vector obtained by columnwise reordering the elements of a matrix $\mathbf{V}$.
We refer to the corresponding inexact NP method as NPCG method and we propose to use it for the computation of $\bf^{(k)}$ (step 2).
\\
Let us now denote by $Q^{(k)}(\bf)$ the least squares objective function:
\begin{equation}\label{Q}
    Q^{(k)}(\bf) = \| \mathbf{K} \mathbf{f} - \mathbf{s} \|^2 + \sum_{i=1}^N \lambda_i^{(k)}(\mathbf{L}\mathbf{f})_i^2
\end{equation}
and by $\mathcal{A}(\bf)$ the set of indices \cite{vog02}
  \begin{equation*}
    \mathcal{A}(\bf) = \Big\{i \; |\; 0\leq f_i \leq \varepsilon \text{ and
    } (\nabla Q)_i>0\Big\}, \quad \varepsilon = \min \{ \overline{\varepsilon},\|\mathbf{f}-[\mathbf{f}- \nabla Q]^+\| \}
  \end{equation*}
where
$\overline{\varepsilon}$ is a small positive parameter and $[\cdot]^+$ denotes the projection on the positive orthant.
Finally, let $\mathbf{E}$ and $\mathbf{F}$ denote the diagonal matrices \cite{vog02}
such that
  \begin{align*}
    \{ \mathbf{E}(\mathbf{f}) \}_{ii} &= \left\{
                         \begin{array}{ll}
                        1, & i \notin \mathcal{A}(\mathbf{f}); \\
                        0, & i \in \mathcal{A}(\mathbf{f});
                      \end{array}
                    \right. \\
  \mathbf{F}(\mathbf{f}) &= \mathbf{I}-\mathbf{E}(\mathbf{f})  .
  \end{align*}
The NPCG method for the minimization of $Q^{(k)}(\bf)$ under nonnegativity constraints can be stated formally as follows.
\smallskip
\paragraph{Algorithm 1: NPCG method. \label{PNCG}}\hfill\\
\vspace*{-.5cm}
\begin{itemize}
  \item[]  \textbf{Initialization}: choose $\bf^{(0)}$ and set $\ell=0$.
  \item[] \textbf{repeat}
  \begin{enumerate}
     \item compute the index subset $\mathcal{A}^{(\ell)}$ and the matrices $\mathbf{E}^{(\ell)}$ and $\mathbf{F}^{(\ell)}$;
     \item determine the search direction $\mathbf{d}^{(\ell)}$ by solving, with the CG method, the linear system
     \begin{equation}\label{cg_pn}
        \big(\mathbf{E}^{(\ell)}\nabla^2Q^{(k)}(\bf^{(\ell)})\mathbf{E}^{(\ell)}+\mathbf{F}^{(\ell)}\big)\mathbf{d}=-\nabla Q^{(k)}(\bf^{(\ell)});
     \end{equation}
     \item determine a step-length $\alpha^{(\ell)}$ satisfying the Armijo rule along the projection arc \cite{ber99};
     \item compute $\bf^{(\ell+1)}=[\bf^{(\ell)}+\alpha^{(\ell)}\mathbf{d}^{(\ell)}]^+$;
     \item set $\ell=\ell+1$
  \end{enumerate}
  \textbf{until} a stopping criterion is satisfied.
\end{itemize}
\smallskip

Summarizing, the 2DUPEN algorithm can be stated formally as follows.
\smallskip
\paragraph{Algorithm 2: 2DUPEN method. \label{UPEN}}\hfill\\
\vspace*{-.5cm}
\begin{itemize}
  \item[] \textbf{Initialization}: choose $\beta_0$, $\beta_c$ and $\beta_p$;
  compute an approximated solution $\mathbf{f}^{(0)}$ to the problem
   \begin{equation}
   \min\limits_{\mathbf{f}\geq 0} \,  \| \mathbf{K} \mathbf{f} - \mathbf{s} \|^2 
   \end{equation}
   by applying a few iterations of the GP method;
   set $k=0$.
  \item[] \textbf{repeat}
  \begin{enumerate}
     \item compute $\lambda_i^{(k)} = \displaystyle{\frac{\|\mathbf{K} \mathbf{f}^{(k)} - \mathbf{s} \|^2}
     {N\left ( \beta_0 +\beta_p
\underset{\substack{\mu \in I_i}}\max \,( \mathbf{p}_{\mu}^{(k)})^2
 + \beta_c \underset{\substack{\mu \in I_i }} \max \, (\mathbf{c}_{\mu}^{(k)})^2\right )}}$;
     \item calculate
     $\mathbf{f}^{(k+1)}$ by solving, with the NPCG method, the constrained minimization problem
     \begin{equation}\label{}
        \min\limits_{\mathbf{f}\geq 0} \, \left\{\| \mathbf{K} \mathbf{f} - \mathbf{s} \|_2^2 +  \sum_{i=1}^N \lambda_i^{(k)}(\mathbf{L}\mathbf{f})_i^2\right\};
     \end{equation}

    \item set $k=k+1$;
  \end{enumerate}
  \textbf{until} a stopping criterion is satisfied.
\end{itemize}

%
We point out that Algorithm 2 can be easily modified to handle the case  $f \geq \rho$, $\rho  \neq 0$
or the case of no constraints. The former case uses the version of NP for bound constrained problems described in \cite{ber82Siam}
where the index subset $\mathcal{A}$ is
  \begin{equation*}
    \mathcal{A}(\bf) = \Big\{i \; |\; \rho\leq f_i \leq \rho+\varepsilon \text{ and
    } (\nabla Q)_i>0\Big\} .
  \end{equation*}
In the case of no constraints, the CG algorithm is applied instead of the NP method.

%
%
%
%
\section{Numerical Results \label{Num}}
In this section we report the results obtained by the 2DUPEN method defined by \textbf {Algorithm 2} with
simulated and real 2D NMR data.
The aim of the experiments is to have a first verification and validation of the proposed algorithm.
Let us firstly describe the experimental setting used in all our numerical experiments.

\subsection{Experimental setting}
The numerical experiments have been executed on a PC with Intel i7 processor (3.4GHz, 16GB RAM) using Matlab R2012a.

The 2DUPEN method has been compared with Tikhonov method where NP has been used to solve the constrained minimization problem:
 \begin{equation}\label{eq:tikh3}
  \begin{split}
    \min_{\mathbf{f}} \left\lbrace \| \mathbf{K} \mathbf{f} - \mathbf{s} \|^2 +  \alpha \|\mathbf{L} \mathbf{f} \|^2 \right\rbrace
     & \text{ s.t.}  \; \mathbf{f} \geq 0 .
  \end{split}
 \end{equation}
The spatially adapted regularization parameters $\lambda_i^{(k)}$, used by the 2DUPEN method, are computed by \eqref{eq:lambda2} where
the indices of the sets $I_i$ are relative to a $3 \times 3$ mask centered at the $i-th$
point of coordinates $(j, k)$:
$$\underset{\substack{\nu \in I_i}}\max \, \mathbf{p}_{\nu}^2
= \underset{\substack{j-1 \leq \ell \leq j+1 \\ k-1 \leq \mu \leq k+1}} \max (P^2_{\ell, \mu}), \ \ \
\underset{\substack{\nu \in I_i}}\max \, \mathbf{c}_{\nu}^2
= \underset{\substack{j-1 \leq \ell \leq j+1 \\ k-1 \leq \mu \leq k+1}}  \max (C^2_{\ell, \mu}) $$
and the matrices $\mathbf{P}$ and $\mathbf{C}$ are computed using
forward and central finite differences, respectively.
The initial approximate solution $\bf^{(0)}$ is computed by stopping the GP iterations as soon as:
$$  \| \mathbf{r}^{k} \| -\| \mathbf{r}^{k-1} \| \leq \textsf{Tol}_\textsf{GP} \| \bs \|$$
where $\mathbf{r}^{k}$ is the residual vector at the $k$-th step and $\textsf{Tol}_\textsf{GP}$
is the relative tolerance parameter of GP method. The maximum number of iterations allowed for GP is
$\textsf{Kmax}_\textsf{GP}=50000$.

The 2DUPEN algorithm stops when the relative distance between two successive iterates becomes
less than a tolerance $\textsf{Tol}_\textsf{UPEN}$ or after a maximum of $\textsf{Kmax}_\textsf{UPEN}$ iterations.
After a wide experimentation, the values $\textsf{Tol}_\textsf{UPEN}=10^{-3}$, $\textsf{Kmax}_\textsf{UPEN} = 500$  have been fixed
and used in this set of preliminary experiments.

The iterations of the NP method used in both 2DUPEN and Tikhonov
methods have been stopped on the basis of the relative decrease of the objective function \eqref{Q}:
$$\frac{Q^{(k)}(\bf)- Q^{(k-1)}(\bf)}{Q^{(k)}(\bf)} < \textsf{Tol}_\textsf{NP}$$
and the inner linear system \eqref{cg_pn} is solved by the CG method  with relative tolerance $\textsf{Tol}_\textsf{CG}$.
A maximum of $N_xN_y$ iterations have been allowed for both NP and CG.

\subsection{Experiments on Simulated Data}
In this paragraph we consider two test problems obtained by
inverting $M_1 \times M_2$ simulated IR-CPMG data ($\bs$)
artificially synthesized from two different model distributions
$\mathbf{f^*}$ of $Nx \times Ny$ relaxation times, to which noise is added.

In the first test problem (P1) the exact solution $\mathbf{f^*}$ has $Nx \times Ny = 64 \times 64$
relaxation times (see figure \ref{fig:F1}).
This test distribution is characterized by two well separated peaks over a zero flat area.

In the second test problem (P2) the exact solution $\mathbf{f^*}$
has $Nx \times Ny = 96 \times 96$ relaxation times (see figure \ref{fig:F1_p}).
This test distribution still presents separated peaks but there is also a quite large non-zero flat area.
This second test problem is
more similar to the usual experimental conditions.

The noisy data are defined as
$\bs = \mathbf{y} + \be$ where $\mathbf{y}= \mathbf{K} \mathbf{f^*}$
represents
a $M_1 \times M_2$ noiseless signal with $M_1= M_2 = 128$.
In our experiments we define the noise vector $\be$ of level $\delta$ as
$ \be = \delta \boldsymbol \eta$ where $\delta > 0$
and $\boldsymbol \eta$ is a normal random Gaussian vector such that $\| \boldsymbol \eta \| = 1$.
By changing $\delta$  we evaluate the performance of the algorithm, computing the following error parameters:
$$
\text{Err}=\frac{\| \mathbf{f}-\mathbf{f^*} \|}{\| \mathbf{f^*}  \|}, \ \ \ \hbox{Relative Error}
$$
$$
\text{Res} = \| \mathbf{K} \mathbf{f} - \mathbf{s} \|, \ \ \ \hbox{Residual  Norm}
$$
$$
\chi = \frac{\| \mathbf{f}-\mathbf{f^*} \| }{\sqrt{N}}, \  \ \ \hbox{Mean Squared Error}
$$
where $\mathbf{f}$ represents the computed distribution.
The values $\beta_p=\beta_c=1$ and $\beta_0=10^{-6}$  have been used in \eqref{eq:lambda2}.

The experiments consist in using  the 2DUPEN method to reconstruct the model distribution with noisy data where
$\| \be \| = 10^{-3}, 10^{-2}$, and  $10^{-1}$. The results are then compared to the best reconstruction obtained by
solving \eqref{eq:tikh3} using  the optimal scalar regularization parameter (Tikhonov method).
In this case the regularization parameter is obtained by a posteriori minimization  of the relative error.

We consider here the test problem P1. In table \ref{tab:T1} we report the error parameters ($\chi$, Err), the residual norm values (Res), the number of 2DUPEN iterations
(k\_upen) and the number of CG iterations (it\_cg).

Rows 2 and 3 ($\| \be \| = 10^{-2} $ and $10^{-1}$) in table \ref{tab:T1}  are obtained  setting  the parameters:
 $\textsf{Tol}_\textsf{GP}=10^{-2}$, $\textsf{Tol}_\textsf{NP}=10^{-6}$ and $\textsf{Tol}_\textsf{CG}=10^{-3}$.
The results in row 1 ($\| \be \| = 10^{-3}$) required smaller tolerance values:
$\textsf{Tol}_\textsf{GP}=10^{-3}$, $\textsf{Tol}_\textsf{NP}=10^{-8}$ and $\textsf{Tol}_\textsf{CG}=10^{-4}$.
This caused an increase in the number of inner CG iterations.

In figure \ref{fig:F1} we show the 2DUPEN distribution and the
locally adapted regularization
matrix $\boldsymbol \Lambda$ containing the values of the regularization parameters $\lambda_i$ corresponding to the $T_1-T_2$ map of the 2DUPEN reconstruction, in the case $\| \be \|=10^{-2}$.
The behavior of the relative error in this case is plotted in figure \ref{fig:Errplot}(a).
We observe its fast decrease in the first steps and very small changes as the iterations proceed. A similar behavior can also be observed in the residual norms plotted in figure \ref{fig:Errplot}(b). By observing the Res values in table \ref{tab:T1}
we see that 2DUPEN method computes  a very good estimate of the norm of the noise vector ($|| \be || $). Hence we can conclude that 2DUPEN iterations improve the initial estimate of the noise norm obtaining very accurate results.
The number of CG iterations can be very large especially with low noise since smaller tolerances are required.
By observing the 2D surface of the locally
adapted regularization parameters (figure \ref{fig:F1}, $\log(\boldsymbol \Lambda)$)
we notice very large values, correspondent to  the flat regions and steeply decreasing values, related
to the  peaks of the map.
The locally adapted regularization parameters $\boldsymbol{\Lambda}$ computed by 2DUPEN
have values in the range $[3.82E-3, 2.64E+5]$.
%
\begin{table}[!htbp]
\centering
\begin{tabular}{ccccc}
\hline
 $|| \be || $ & $\chi$ & Err & Res & k\_upen(it\_cg) \\
\hline
$10^{-3}$ & 9.1183E-5 & 4.7626E-2 & 9.9827E-4 & 13(2775207) \\
$10^{-2}$ & 1.4617E-4 & 7.6344E-2 & 9.9881E-3 & 16(999905) \\
$10^{-1}$ & 1.9697E-4 & 1.0288E-1 & 9.9912E-2 & 8(114686) \\
\hline
\end{tabular}
\caption{Test problem P1: error parameters and iterations obtained by 2DUPEN with the P1 simulated data.}\label{tab:T1}
\end{table}
%
%
The same experiment is repeated using the scalar regularization parameter $\alpha$, as required by the Tikhonov method \eqref{eq:tikh3}.
Table \ref{tab:T2} reports the results obtained by the optimal regularization parameter, computed
by a posteriori minimization of the relative error.
\begin{table}[!h]
\centering
\begin{tabular}{ccccccc}
\hline
$|| \be \| $ & $\chi$ & Err & Res & it\_cg & $\alpha$ \\
\hline
$10^{-3}$ & 2.6693E-4 & 1.3942E-1 & 1.0125E-3 & 31689 & 6.0520E-6  \\
$10^{-2}$ & 2.7516E-4 & 1.4423E-1 & 9.9994E-2 & 19834 & 2.4822E-3  \\
$10^{-1}$ & 2.7613E-4 & 1.6391E-1 & 1.0889E-1 & 19711 & 2.4382E-1  \\
\hline
\end{tabular}
\caption{Test problem P1: error parameters and iterations obtained by Tikhonov reconstruction with optimal regularization parameter $\alpha$.}\label{tab:T2}
\end{table}
In the rest of figure \ref{fig:F1} , we report the maps of 2DUPEN and Tikhonov reconstructions in the case $\|\mathbf{e}\|=10^{-2}$.
The best Tikhonov reconstruction, shown in figure \ref{fig:F1} (bottom right), is obtained
with a constant regularization parameter $\alpha =2.48E-3$ which is smaller than  the smallest value of
 the 2DUPEN regularization parameter $\boldsymbol\Lambda$ ($3.82E-03$), allowing a good reconstruction of the peaks. However, in this case, also the zero flat areas are well reconstructed, as shown in the sum projection plots along the 1D $T_1$ and 1D $T_2$ distributions
 (figure \ref{fig:T1T2}). This proves that, in test problem P1, 2DUPEN and Tikhonov method coupled with an efficient method for automatically
 computing the regularization parameter, can produce comparable reconstructions.
\begin{figure}[!htbp]
 \begin{center}
 \hspace*{-.8mm}
 \begin{tabular}{cc}
   $\bf^*$ & $log(\Lambda)$ \\
   {\includegraphics[width=.45\textwidth]{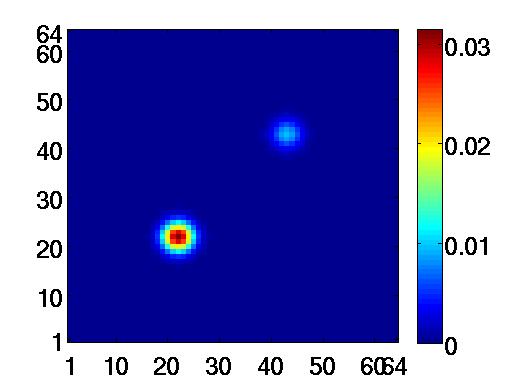}}  &
   {\includegraphics[width=.45\textwidth]{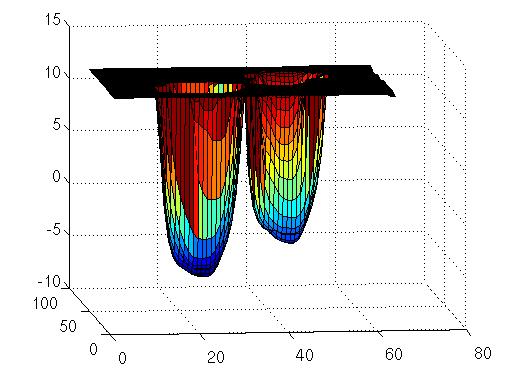}} \\
   {\textsf{UPEN}} & {\textsf{Tikhonov ($\alpha=2.482e-3$)}} \\
   {\includegraphics[width=.45\textwidth]{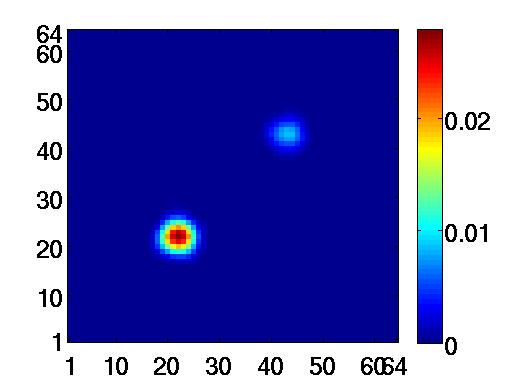}}  &
   {\includegraphics[width=.45\textwidth]{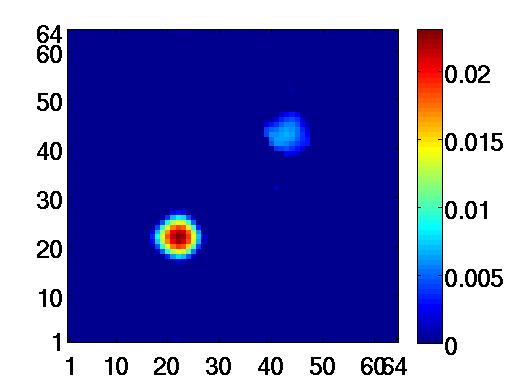}}\\

 \end{tabular}
 \caption{Test problem P1. Top row from left to right: $T_1-T_2$ map of the original model distribution (exact solution $\bf^*$),
 2DUPEN regularization matrix in log scale ($log(\boldsymbol{\Lambda})$).
 Bottom row: $T_1-T_2$ maps reconstructed by 2DUPEN and Tikhonov methods, respectively.}\label{fig:F1}
\end{center}
\end{figure}
\begin{figure}[!htbp]
 \begin{center}
 \hspace*{-.8mm}
 \begin{tabular}{cc}
   $\bf^*$ & $log(\Lambda)$ \\
   {\includegraphics[width=.45\textwidth]{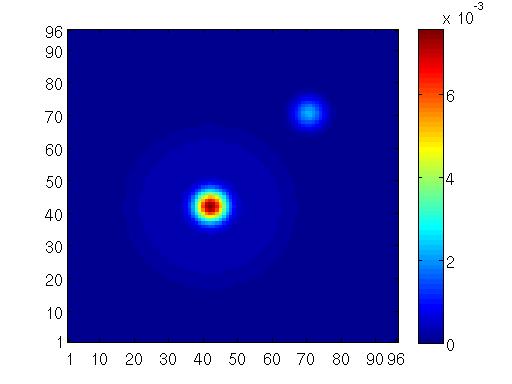}}  &
   {\includegraphics[width=.45\textwidth]{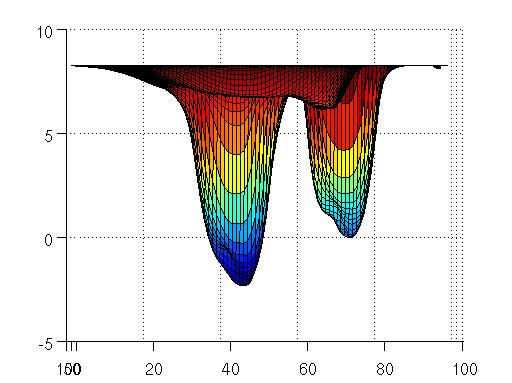}} \\
   {\textsf{UPEN}} & {\textsf{Tikhonov ($\alpha=4.125e-2$)}} \\
   {\includegraphics[width=.45\textwidth]{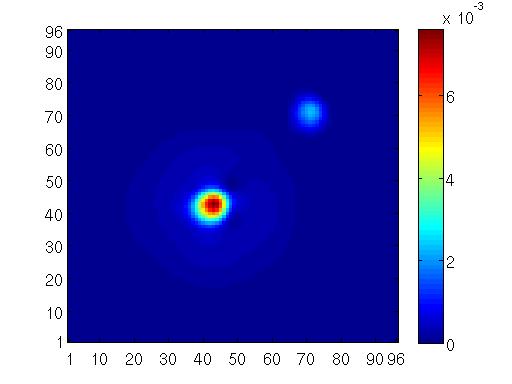}}  &
   {\includegraphics[width=.45\textwidth]{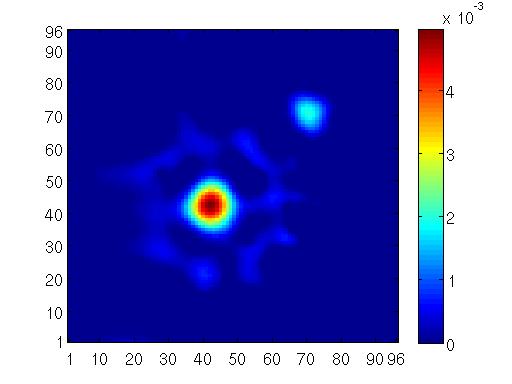}}\\
 \end{tabular}
 \caption{Test problem P2. Top row from left to right: $T_1-T_2$ map of the original model distribution
 (exact solution $\bf^*$),
 2DUPEN regularization matrix in log scale ($log(\boldsymbol{\Lambda})$).
 Bottom row: $T_1-T_2$ maps reconstructed by 2DUPEN and Tikhonov method, respectively.}\label{fig:F1_p}
\end{center}
\end{figure}
\begin{figure}[!htbp]
 \centering
 \begin{tabular}{cc}
 (a) & (b) \\
 \includegraphics[width=6cm,height=5cm,bb=0 0 512 384]{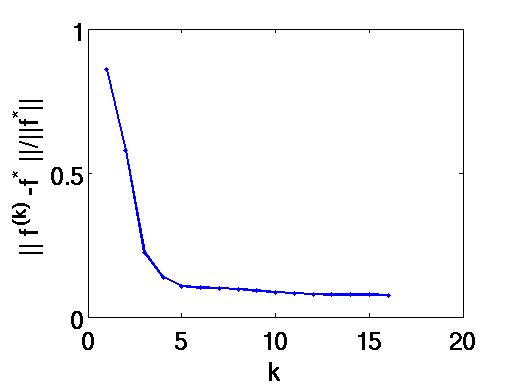} &
 \includegraphics[width=6cm,height=5cm,bb=0 0 512 384]{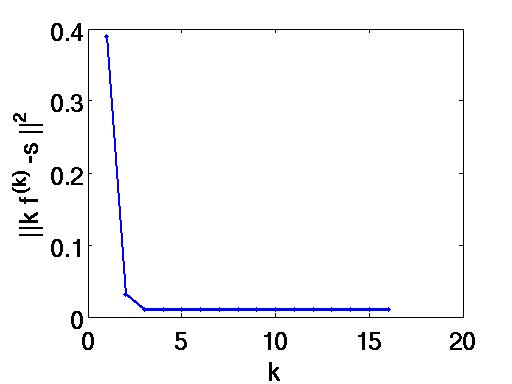} \\
 \end{tabular}
 \caption{Test problem P1: Relative Error (Err) and Residual norm (Res) per iteration ($\| \be \|=1.E-2$).}
 \label{fig:Errplot}
\end{figure}
\begin{figure}[!htbp]
 \begin{center}
 \hspace*{-.8mm}
 \begin{tabular}{cc}
   {\includegraphics[width=.45\textwidth]{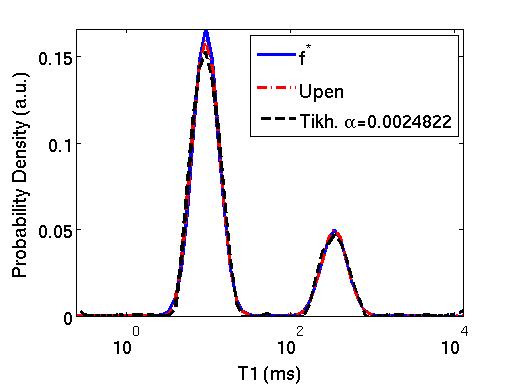}}  &
   {\includegraphics[width=.45\textwidth]{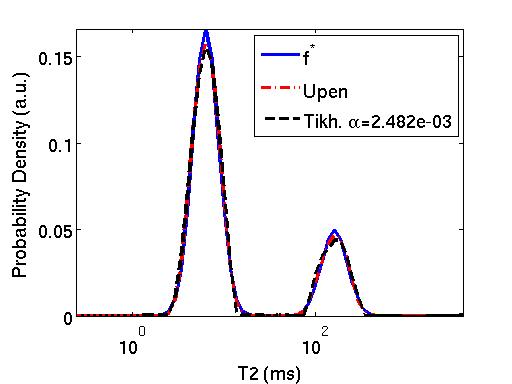}} \\
 \end{tabular}
 \caption{Test problem P1. Sum projection plots along the $T_1$ and $T_2$ dimensions, of the exact solution $\bf^*$ (blue line),
 2DUPEN (red dash-dot line) and Tikhonov (black dash line ) reconstructions, respectively ($\| \be \|=1.E-2$).}\label{fig:T1T2}
\end{center}
\end{figure}
The second test problem P2 is considered hereafter. Analogously to the previous case, the error parameters obtained with 2DUPEN are reported in table \ref{tab:T1_p}.
Also in this case we observe that the residual norm value obtained in column Res is a good estimate of the noise norm $\| \mathbf{e} \|$.

The 2D surface of the spatially adapted regularization parameters (figure \ref{fig:F1_p}, $\log(\boldsymbol \Lambda)$)
has the largest values in correspondence to the zero flat regions, mildly deceasing values in flat non zero areas
and steeply decreasing values in correspondence to the image peaks. This produces a good reconstruction of the
different features.
\begin{table}[!h]
\centering
\begin{tabular}{ccccc}
\hline
 $|| \be || $ & $\chi$ & Err & Res & k\_upen(it\_cg) \\
\hline
$10^{-3}$ & 7.4027E-5 & 1.5174E-1 & 9.9780E-4 & 16(4012669) \\
$10^{-2}$ & 6.5259E-5 & 1.3377E-1 & 9.9893E-3 & 11(450895) \\
$10^{-1}$ & 8.1657E-5 & 1.6739E-1 & 9.9913E-2 & 38(566106)  \\
\hline
\end{tabular}
\caption{Test problem P2: error parameters and iterations obtained by 2DUPEN with simulated data.}\label{tab:T1_p}
\end{table}
The same experiment is repeated using the scalar regularization parameter $\alpha$ as required by   the Tikhonov method defined in
\eqref{eq:tikh3}.
Table \ref{tab:T2_p} reports the results obtained by the optimal regularization parameter, computed
by a posteriori minimization of the relative error.
\begin{table}[!h]
\centering
\begin{tabular}{ccccccc}
\hline
$|| \be \| $ & $\chi$ & Err & Res & $\alpha$ & it\_cg  \\
\hline
$10^{-3}$ &  1.3963E-4 & 2.8622E-1 & 3.1594E-2 & 1.1938E-1 & 55182 \\
$10^{-2}$ &  1.5091E-4 & 3.0935E-1 & 9.9943E-2 & 4.1246E-1 & 33030 \\
$10^{-1}$ &  1.7304E-4 & 3.2989E-1 & 9.9998E-2 & 2.7567E-2 & 43410 \\
\hline
\end{tabular}
\caption{Test Problem P2: error parameters and iterations obtained by Tikhonov reconstruction with optimal regularization parameter $\alpha$.}\label{tab:T2_p}
\end{table}
The best Tikhonov reconstruction, shown in figure \ref{fig:F1_p} (bottom right), is obtained
 with a constant regularization parameter $\alpha =4.1246E-1$ which is larger than the smallest
 component of $\boldsymbol\Lambda$ ($\lambda_{min}=2.7567E-02$) causing an underestimate of the highest peak.

On the other hand reconstructing with $\alpha =2.7567E-02$ we obtain an improvement in the reconstruction of the peak
but also a larger relative error Err$=3.2989E-1$ due to the increased oscillations in the flat areas,
as shown in figures \ref{fig:T1T2_p} and \ref{fig:T1T2_p_1} relative to the sum projection plots
along the 1D $T_1$ and 1D $T_2$ distributions.
 \begin{figure}[!h]
  \begin{center}
  \begin{tabular}{cc}
    {\includegraphics[width=.45\textwidth]{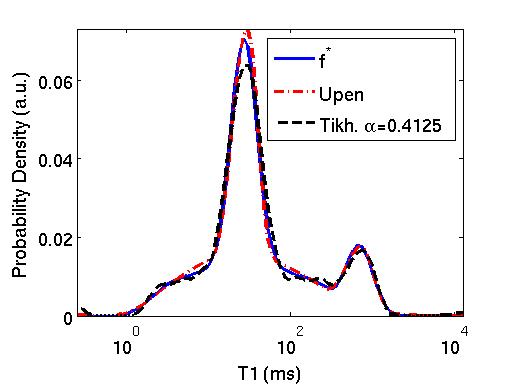}}  &
    {\includegraphics[width=.45\textwidth]{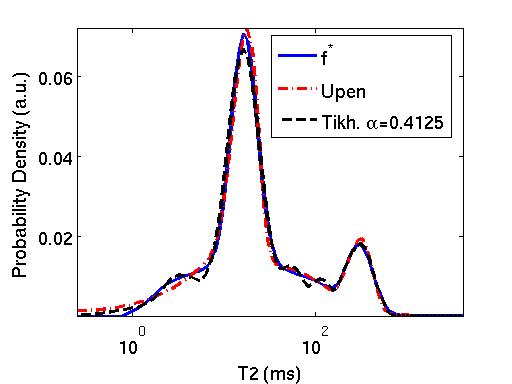}} \\
  \end{tabular}
  \caption{Test problem P2. Projections along $T_1$, $T_2$ of exact solution $\bf^*$ (blue line),
  2DUPEN (red dash-dot line) and Tikhonov  (black dash line ) ($\| \be \|=1.E-2$).}
  \label{fig:T1T2_p}
 \end{center}
 \end{figure}
\begin{figure}[!h]
  \begin{center}
  \begin{tabular}{cc}
    {\includegraphics[width=.45\textwidth]{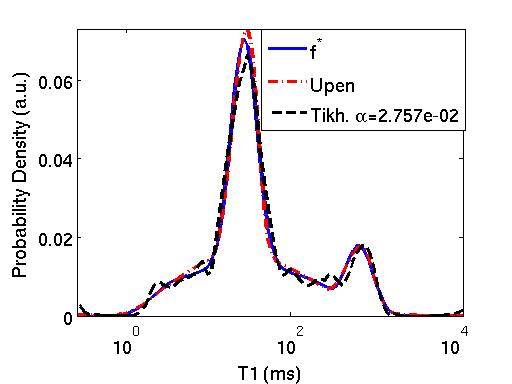}}  &
    {\includegraphics[width=.45\textwidth]{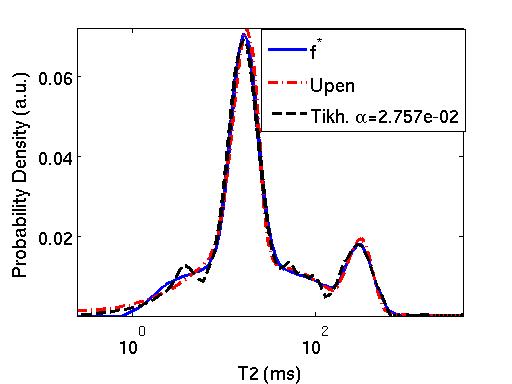}} \\
  \end{tabular}
  \caption{Test problem P2. Projections along $T_1$, $T_2$ of exact solution $\bf^*$ (blue line),
  2DUPEN (red dash-dot line) and Tikhonov (black dash line ) ($\| \be \|=1.E-2$).} \label{fig:T1T2_p_1}
 \end{center}
 \end{figure}
%

In this case the reconstructions obtained by 2DUPEN and Tikhonov are never comparable, since Tikhonov method
is not able to approximate correctly both the high peaks and the non-zero flat regions.
It appears that 2DUPEN can better reproduce the underlying exact distribution expecially in presence of peaks and non nonzero flat regions.
More studies on this will be performed.
This tests verified the correctness of the implemented algorithm as well as the effectiveness
of the 2DUPEN principle for the selection of the locally adapted regularization parameters.

\subsection{Experiments on Real Data}
We now present the results obtained on real NMR data. A sample was prepared by
filling a 10 mm external diameter glass NMR tube with 6 mm of egg yolk.
The tube was sealed with Parafilm, and then at once measured.
NMR measurements were performed at 25 \textcelsius \;
by a homebuilt relaxometer based on a PC-NMR portable NMR console (Stelar, Mede, Italy) and a 0.47 T Joel electromagnet.

All relaxation experimental curves were acquired using phase-cycling procedures.
The $\pi/2$ pulse width was of 3.8 $\mu s$ and the relaxation delay (RD) was
set to a value greater than 4 times the maximum $T_1$ of the sample. In all experiments RD was equal to 3.5 s.

For the 2D measurements, longitudinal-transverse relaxation curve ($T_1$-$T_2$) was acquired by an
IR-CPMG pulse sequence
(RD - $\pi_x$ - $T_I$ - $(\pi/2)_x$ - TE/2-[$\pi_y$-TE/2-echo acquisition-TE/2]$_\text{NE}$).
The $T_1$ relaxation signal was acquired with 128 inversion times ($T_I$) chosen in geometrical progression
from 1 ms up to 2.8 s, with $NE=1024$ (number of acquired echos, echo times $TE= 500 \mu s$) on each CPMG,
and number of scans equal to 4.
All curves were acquired using phase-cycling procedures.

As CPMG data blocks of an IR-CPMG sequence can have thousands of points,
to avoid excessive computation time it may be necessary to reduce the number of points of each CPMG data blocks.
If the noise is additive, random, and, approximately, normally distributed, and if systematic
data errors are smoothly varying with time, then averaging data points into sufficiently
narrow windows does not change the result with respect to that obtained by using all points.
In some cases, the windowed number of points for computation can be reduced by orders of magnitude.
In this work, the windowing was implemented following the method described in \cite{Borgia2000}.
After the application of the windowing, the points of CPMG blocks were non-equally spaced and reduced to a number of 146.

For the UPEN inversion, in order to respect approximatively the same ratio existing between $M_1$ and $M_2$,
the values for $N_x=64$ and $N_y=73$
were chosen and the values $\textsf{Tol}_\textsf{GP}=0.01$,
$\textsf{Tol}_\textsf{NP}=10^{-4}$ and $\textsf{Tol}_\textsf{CG}=0.1$ were fixed.

For this test problem with real data an exact solution is not available but from several experimental studies
we know that the 2DUPEN method with $ \beta_p=5\cdot10^{-2}$, $\beta_c=2\cdot10^{-2} $ and $\beta_0 = 5\cdot10^{-7}$
 correctly computes the position and value of
the two peaks and that there is a smooth bending  between them.

The $T_1-T_2$ maps obtained from the 2DUPEN inversion is shown in figure \ref{fig:2r} (top row);
figure \ref{fig:2r} also shows the $T_1-T_2$ maps obtained by Tikhonov with
$\alpha=0.1$ (middle line) and $\alpha=20$ (bottom line).
Figure \ref{fig:3r} shows, in log-scale, the regularization matrix
$\boldsymbol{\Lambda}$.
%
\begin{figure}[!htbp]
 \begin{center}
 \hspace*{-.8mm}
 \begin{tabular}{cc}
   \multicolumn{2}{c}{\textsf{2DUPEN}}\\
   {\includegraphics[width=.45\textwidth]{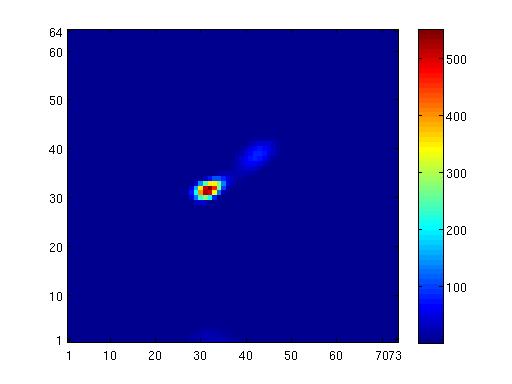}}  &
   {\includegraphics[width=.45\textwidth]{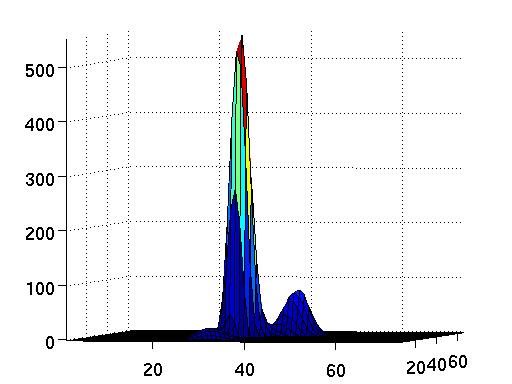}} \\
   \multicolumn{2}{c}{\textsf{Tikhonov ($\alpha=0.1$)}}\\
   {\includegraphics[width=.45\textwidth]{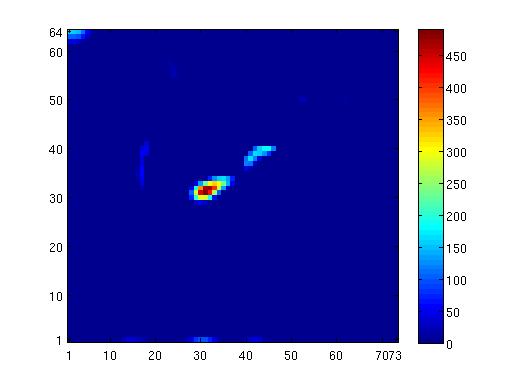}}  &
   {\includegraphics[width=.45\textwidth]{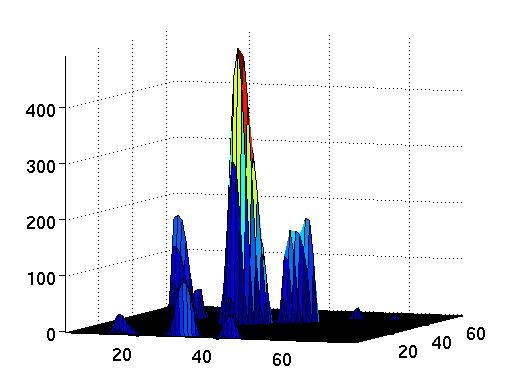}} \\
   \multicolumn{2}{c}{\textsf{Tikhonov ($\alpha=20$)}}\\
   {\includegraphics[width=.45\textwidth]{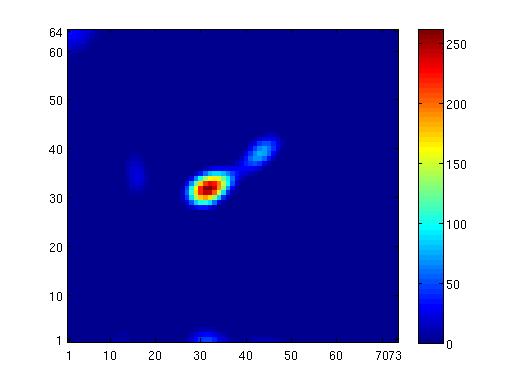}}  &
   {\includegraphics[width=.45\textwidth]{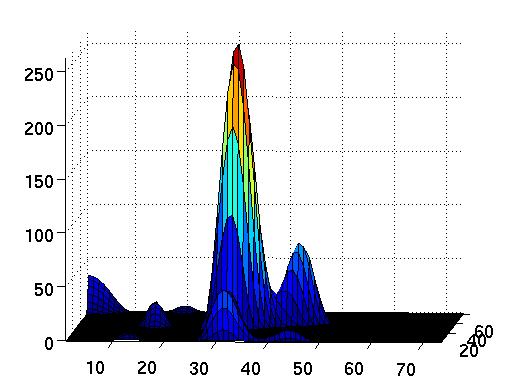}}
 \end{tabular}
 \caption{$T_1-T_2$ maps (left) and 3D distributions (right),
 obtained by the 2DUPEN (top line) and Tikhonov (middle and bottom lines) methods.}\label{fig:2r}
\end{center}
\end{figure}
\begin{figure}[!htbp]
 \begin{center}
  \begin{tabular}{cc}
   \includegraphics[width=.45\textwidth]{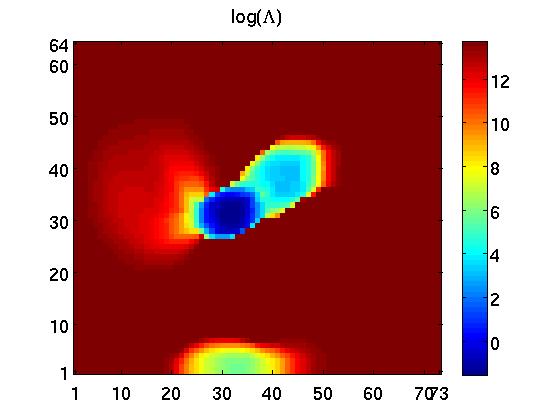}
     \includegraphics[width=.45\textwidth]{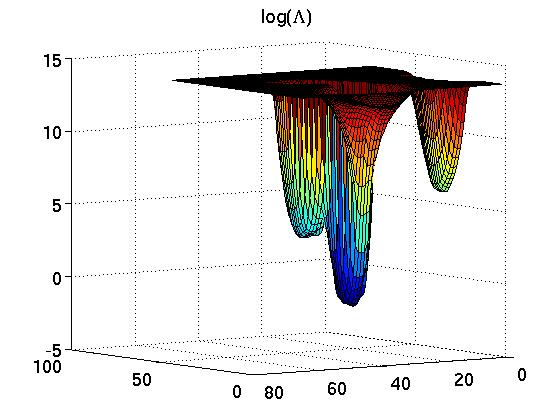}
      \end{tabular}
 \caption{Regularization matrix $\log(\boldsymbol{\Lambda})$ of the 2DUPEN method.}\label{fig:3r}
\end{center}
\end{figure}
\begin{figure}[!htbp]
 \begin{center}
 \hspace*{-.8mm}
 \begin{tabular}{cc}
   {\includegraphics[width=.45\textwidth]{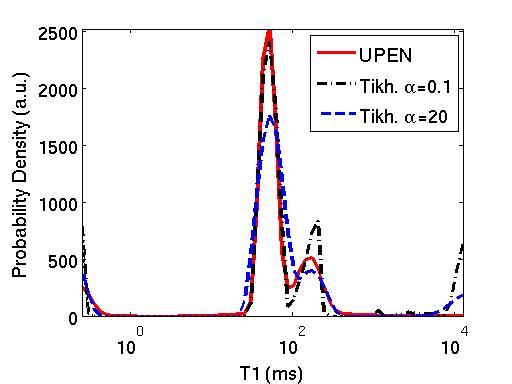}}  &
   {\includegraphics[width=.45\textwidth]{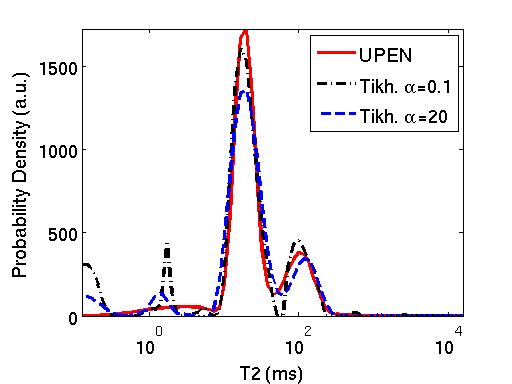}}
 \end{tabular}
 \caption{Sum projections along the $T_1$ and $T_2$ dimensions of the 2DUPEN (red line) and Tikhonov (blue dash and black dash dotted line)
reconstructions.}\label{fig:4r}
\end{center}
\end{figure}
\begin{figure}[!htbp]
 \begin{center}
 \hspace*{-.8mm}
 \begin{tabular}{cc}
   \multicolumn{2}{c}{\textsf{2DUPEN}}\\
   \multicolumn{2}{c}{\includegraphics[width=.45\textwidth]{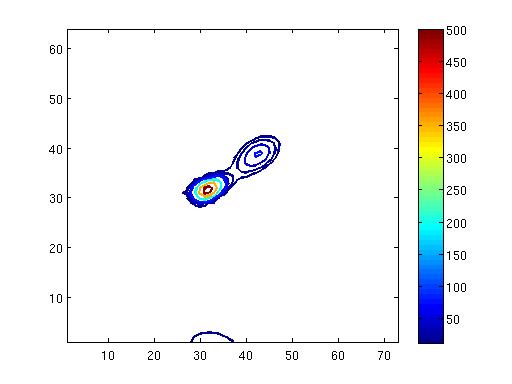}}  \\
   {\textsf{Tikhonov ($\alpha=0.1$)}} & {\textsf{Tikhonov ($\alpha=20$)}} \\
   {\includegraphics[width=.45\textwidth]{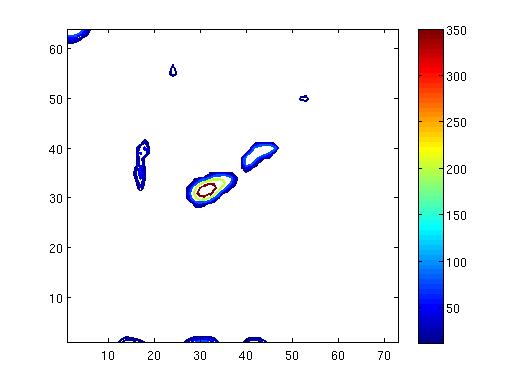}}  &
   {\includegraphics[width=.45\textwidth]{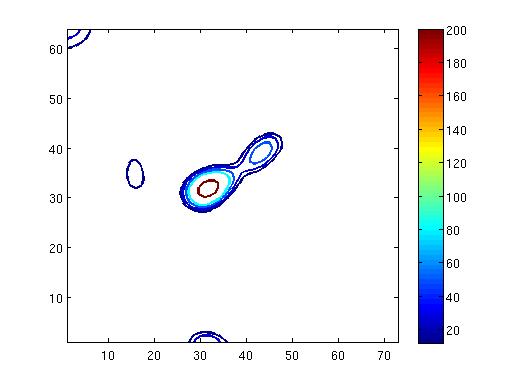}}\\
 \end{tabular}
 \caption{Contour plot of the 2DUPEN (top line) and Tikhonov (bottom line) reconstructions.}\label{fig:5r}
\end{center}
\end{figure}
The comparison with Tikhonov regularization shows that no values of the regularization parameter $\alpha$
allows us to
reconstruct all the features of the 2DUPEN  reconstruction.
Although a quite correct location of the highest peak can be obtained for $\alpha \in [0.01, 50]$
its height and the values in flat regions are  not always well reconstructed.

As expected, different values of the Tikhonov regularization parameter $\alpha$ give more accurate
reconstruction of different features.
For axample, we see that
 $\alpha = 0.1$  gives a better reconstruction of the highest peak both in  $T_1$
 and $T_2$ projections (figure \ref{fig:4r})
 but spurious oscillations appear in the flat regions and the separation between the peaks increases.
 On the other hand if $\alpha=20$ we obtain a better approximation of
 the lower peak  in both $T_1$ and $T_2$ projections, but the highest peak is excessively damped.
 Comparing the contour maps of the 2DUPEN and Tikhonov $T_1-T_2$ maps (figure \ref{fig:5r})
we observe  the presence of spurious structures in flat areas and the excessive widening  of the highest peak
as $\alpha$ increases.

 It is evident that for small values of $\alpha$, Tikhonov regularization, using a constant regularization parameter,
 gives a high peak with value comparable with the value given by 2DUPEN
 but it tends to break the wide peak or tail into two separate peaks.
 Otherwise, for larger values of $\alpha$,
 Tikhonov regularization tends to excessively broaden the sharp peak.
 On the contrary, 2DUPEN, using variable smoothing,
 is able to recover both the sharp peak and the tail improving the possibility
 of performing reliable quantitative analysis based on the 2D distribution.

%
%
%
\section{Conclusions \label{Concl}}
In this paper, the 2DUPEN method has been presented for the inversion of two-dimensional NMR relaxation data.
2DUPEN automatically computes a distribution of relaxation times and spatially adapted regularization
parameters by iteratively solving a sequence of nonnegatively constrained least squares problems
and by updating the regularization parameters according to the Uniform Penalty principle.
Results of numerical experiments on real and simulated NMR data show that 2DUPEN is effective
and outperforms Tikhonov method with constant regularization parameter.
Future work will be related to the study of optimal parameter settings for  different applications
of multidimensional NMR data reconstruction.

\section*{Acknowledgement}
Investigation supported by University of Bologna (FARB Fund).
The authors would like to thank  Leonardo Brizi and Manuel Mariani for the NMR data acquisition.
\bibliography{biblio_F}
\bibliographystyle{unsrt}

\end{document}